\documentclass[11pt]{amsart}
\usepackage{amsmath,amssymb,amsthm, tikz}
\usepackage{xcolor}

\usepackage{amsmath,amssymb,latexsym,amsthm,enumerate,mathtools,marginfix,tikz}
\usepackage{hyperref}
\usetikzlibrary{decorations.pathmorphing}
\tikzset{snake it/.style={decorate, decoration=snake}}
\usepackage{setspace}

\DeclareUnicodeCharacter{2192}{\to}
\DeclareUnicodeCharacter{3C6}{\phi}
\DeclareUnicodeCharacter{221E}{\infty}
\DeclareUnicodeCharacter{2208}{\in}
\DeclareUnicodeCharacter{2265}{\ge}
\DeclareUnicodeCharacter{3A3}{\Sigma}

\usepackage{tikz, tikz-cd} 
\usepackage{microtype} 
\usepackage[margin=1.5in]{geometry}

\usepackage{amssymb}

\renewcommand{\d}{\partial }

\numberwithin{equation}{section}
\newtheorem{Theorem}{Theorem}[section]

\newtheorem{Prop}[Theorem]{Proposition} 
\newtheorem{Corollary}[Theorem]{Corollary} 
\newtheorem{Lemma}[Theorem]{Lemma}
\newtheorem{Definition}[Theorem]{Definition}
\newtheorem*{Main Theorem}{Main Theorem}
\newtheorem*{conj}{Conjecture}

\theoremstyle{definition}

\newtheorem{A special case}{A special case}[Theorem]
\newtheorem*{Coarse van Kampen obstruction class}{Coarse van Kampen obstruction class}
\newtheorem*{Coarse cohomology of the configuration space}{Coarse cohomology of the configuration space}

\newtheorem*{Cannon Conjecture}{Cannon Conjecture}
\newtheorem*{Generalized Cannon Conjecture}{CAT(0) Cannon Conjecture}

\theoremstyle{remark}
\newtheorem{Remark}[Theorem]{Remark}
\newtheorem{Example}[Theorem]{Example}

\newcommand{\rr}{\mathbb{R}}
\newcommand{\zz}{\mathbb{Z}}
\newcommand{\nn}{\mathbb{N}}

\newcommand{\ee}{\mathbb{E}}

\newcommand{\cS}{\mathcal{S}}
\newcommand{\cC}{\mathcal{C}}

\newcommand{\cD}{\mathcal{D}}
\newcommand{\symdiff}{\mathbin{\triangle}}

\DeclareMathOperator\diam{diam}

\DeclareMathOperator\myc{C}
\DeclareMathOperator\myh{H}
\DeclareMathOperator\Hom{Hom}

\DeclareMathOperator\cat{CAT(0)}
\DeclareMathOperator\Isom{Isom}
\DeclareMathOperator\Stab{Stab}

\newcommand{\comment}[1]{}

\newcommand{\C}[1][*]{\myc^{#1}}

\newcommand{\Cx}[1][*]{\myc{\mathrm X}^{#1}}
\newcommand{\cx}[1][*]{\myc{\mathrm X}_{#1}}

\newcommand{\clf}[1][*]{\myc_{#1}^{lf}}
\newcommand{\cf}[1][*]{\myc_{#1}}

\newcommand{\Hx}[1][*]{\myh{\mathrm X}^{#1}}
\newcommand{\hx}[1][*]{\myh{\mathrm X}_{#1}}

\newcommand{\rh}[1][*]{\tilde{\myh}_{#1}}

\newcommand{\gs}{\sigma }
\newcommand{\gt}{\tau}

\newcommand{\gD}{\Delta}
\newcommand{\gd}{\delta}

\renewcommand{\ge}{\varepsilon}

\let\oldsubset\subset
\renewcommand{\subset}[1][]{\overset{#1}{\oldsubset}}

\newcommand{\capp}{\mathbin{\frown}}

\let\oldin\in
\renewcommand{\in}[1][]{\overset{#1}{\oldin}}

\let\oldnotin\notin
\renewcommand{\notin}[1][]{\overset{#1}{\oldnotin}}

\DeclarePairedDelimiter\abs{\lvert}{\rvert}
\newcommand{\Supp}[2][]{\abs{#2}_{#1}}

\NewDocumentCommand\ccap{o}{\mathbin{\overset{\mathrm{c}
	\IfNoValueTF{#1}{} {, #1} }
	{\cap}  } }
\NewDocumentCommand\nceq{o}{\mathbin{\overset{\mathrm{c}
	\IfNoValueTF{#1}{} {, #1} }
	{\neq}  } }	
\NewDocumentCommand\cminus{o}{\mathbin{\oset{\mathrm{c}
	\IfNoValueTF{#1}{} {, #1} }
	{-}  } }	
	
\NewDocumentCommand\ceq{o}{\mathbin{\overset{\mathrm{c}
	\IfNoValueTF{#1}{} {, #1} }
	{=}  } }
\NewDocumentCommand\cneq{o}{\mathbin{\overset{\mathrm{c}
	\IfNoValueTF{#1}{} {, #1} }
	{\neq}  } }	
	
\NewDocumentCommand\csubset{o}{\subset[\mathrm{c}\IfNoValueTF{#1}{} {, #1} ]}

\makeatletter
\newcommand{\oset}[3][0ex]{
  \mathrel{\mathop{#3}\limits^{
    \vbox to#1{\\kern-2\ex@
    \hbox{$\scriptstyle#2$}\vss}}}}
\makeatother

\let\oldminus\-

\makeatletter
\newcommand{\DeclareMathActive}[2]{
  
  \expandafter\edef\csname keep@#1@code\endcsname{\mathchar\the\mathcode`#1 }
  \begingroup\lccode`~=`#1\relax
  \lowercase{\endgroup\def~}{#2}
  \AtBeginDocument{\mathcode`#1="8000 }
}

\newcommand{\std}[1]{\csname keep@#1@code\endcsname}
\patchcmd{\newmcodes@}{\mathcode`\-\relax}{\std@minuscode\relax}{}{\ddt}
\AtBeginDocument{\edef\std@minuscode{\the\mathcode`-}}
\makeatother

\begin{document}
\setstretch{1.1}
\title[Cyclic orders and Leary--Minasyan groups]{Cyclic orders and actions of Leary--Minasyan groups on coarse $\mathrm{PD}(n)$ spaces}

\author{Arka Banerjee}
\address{Auburn University, Parker Hall, Auburn, AL 36849, USA}
\email{azb0263@auburn.edu}
\author{Kevin Schreve}
\address{Louisiana State University, Baton Rouge, LA 70803, USA}
\email{kschreve@lsu.edu}
\begin{abstract}
We show that for $n \geq 3$, there are torsion-free CAT(0) groups with visual boundaries that embed into $S^n$ but which are not virtually the fundamental group of a compact aspherical $n+1$-manifold. The groups are the CAT(0) and not bi-automatic groups constructed previously by Leary and Minasyan. 
The obstruction comes from analyzing certain cyclic orders on the boundary of the Bass--Serre tree, in the same manner as Kapovich--Kleiner ruled out actions of Baumslag--Solitar groups on coarse $\mathrm{PD}(3)$ spaces. 
\end{abstract}
\maketitle

\section{Introduction}
This paper is inspired by generalizations of the Cannon Conjecture.

\begin{Cannon Conjecture}
Suppose $G$ is a hyperbolic group with Gromov boundary $\d G \cong S^2$. Then there is a short exact sequence $$1 \rightarrow F \rightarrow G \rightarrow \Gamma \rightarrow 1$$ $G$ where $F$ is finite and $\Gamma$ is a cocompact lattice in $\Isom(\mathbb{H}^3)$. In particular if $G$ is torsion-free then it is the fundamental group of a closed hyperbolic $3$-manifold.
\end{Cannon Conjecture}

There has been progress towards this conjecture, for instance Markovic~\cite{Markovic} showed that the conjecture holds if $G$ is virtually special, see also ~\cite{gm}, ~\cite{GHMOSW}.

Further investigations have explored which hyperbolic and relatively hyperbolic groups with planar boundaries are virtually Kleinian (see~\cite{Haiss},~\cite{HW}, ~\cite{gms}).
Note that there are examples of torsion-free hyperbolic groups with planar boundaries that are not fundamental groups of aspherical $3$-manifolds, but have finite index subgroups which are \cite{KK00},~\cite{HST}.

In fact, as far as we know, there are no counterexamples known to a CAT(0) version of Cannon's conjecture. 

\begin{conj}
Suppose $G$ is a CAT(0) group which admits a planar visual boundary.
Then $G$ has a finite index subgroup which is the fundamental group of an aspherical $3$-manifold. 
\end{conj}

In general, CAT(0) groups can have non-homeomorphic visual boundaries \cite{CK}. It is unknown whether planarity of the visual boundary is an invariant of a CAT(0) group, though Stark and the second author showed that having non-planar graphs in the visual boundary is not an invariant~\cite{SS}. 

There has also been interest in higher-dimensional analogues of Cannon's conjecture. In particular, Bartels, L\"uck and Weinberger showed that if $G$ is a torsion-free hyperbolic group with $\partial G = S^{n-1}$ and $n \geq 6$, then $G$ is the fundamental group of a closed, aspherical $n$-manifold~\cite{BLW}, see also \cite{LT} for the case when $\partial G$ is the $n$-dimensional Sierpinkski space. 
Bregman and Incerti-Medici have recently proved a similar statement for $n = 4$ in the cubulated hyperbolic case \cite{bim}.

The main theorem of this paper is a counterexample to the most ambitious high-dimensional  generalization of a CAT(0) Cannon Conjecture.

\begin{Theorem}\label{t:introthm}
For each $n \geq 3$, there are torsion-free $\cat$ groups with visual boundaries which embed into $S^n$ but which do not (virtually) act geometrically on any contractible $(n+1)$-manifold, in particular, they are not (virtually) the fundamental groups of compact aspherical $(n+1)$-manifolds.
\end{Theorem}

The groups are the $\cat$ but not bi-automatic groups constructed previously by Leary and Minasyan~\cite{Leary-Minasyan}. These groups can be thought of as a CAT(0) analogue of the Baumslag--Solitar groups 
\[
BS(p,q):=\langle a,t\mid ta^pt^{-1}=a^q \rangle.
\]
 $BS(p,q)$ is an HNN extension of $\zz$ in which the stable letter conjugates the index $p$ subgroup to the index $q$ subgroup. Leary--Minasyan groups are analogously HNN extensions of $\zz^n$.
 \subsection{Leary--Minasyan groups} 
 
 Let $L$ be a finitely generated free abelian group, and let $\phi:L'\rightarrow L''$ be an isomorphism between finite-index subgroups of $L$.
 Define a group $G(L,\phi,L')$ as the HNN extension in which case the stable letter conjugates $L'$ to $L''$ via $\phi$:
\[
G(L,\phi,L'):=\langle L,t\mid tct^{-1}=\phi(c),\forall c\in L' \rangle.
\]
In the case when we are given a basis for $L\cong \zz^n$ and $\phi$ is described by a matrix, we simplify the notation slightly.
For $A\in GL(n,\mathbb{Q})$ and $L'$ a finite-index subgroup of $L\cap A^{-1} L=\zz^n\cap A^{-1}L$, we write $G(A,L')$ for the HNN extension defined as above
\[
G(A,L')=\langle L,t\mid tct^{-1}=Ac, \forall c\in L'\rangle.
\] 
\begin{Example}
When $n=1$, these groups coincide with Baumslag--Solitar groups. More precisely, if $n=1$, $A\in GL(1,\mathbb Q)$ is represented by $\frac{q}{p}\in \mathbb Q$, and $L'$ is the index $q$ subgroup of $\zz$, then $G(A,L')=BS(p,q)$. 
\end{Example}
Leary and Minasyan characterized when $G(A,L')$ is a $\cat$ group.
\begin{Theorem}[\cite{Leary-Minasyan}]\label{t:infinite index}    
The group $G(A,L')$ is $\cat$ if and only if $A$ is conjugate in $GL(n,\rr)$ to an orthogonal matrix. Moreover, if $A$ is conjugate in $GL(n,\rr)$ to an orthogonal matrix, then $G(A,L')$ acts geometrically on $T\times \ee^n$ where $T$ is the associated Bass-Serre tree and $T\times \ee^n$ is equipped with the product metric induced from the graph metrics on $T$ and the euclidean metric on $\ee^n$.
\end{Theorem}

Our main theorem is obtained as a corollary of Theorem~\ref{t:infinite index} and the following theorem.
\begin{Theorem}\label{t:main theorem}
    Suppose $A\in GL(n,\mathbb Q)$ has infinite order and is conjugate in $GL(n,\rr)$ to an orthogonal matrix. 
    Then the group $G=G(A,L')$ is CAT(0) and has no finite index subgroup  that acts properly on a coarse $\mathrm{PD}(n+2)$ space. 
\end{Theorem}

A coarse $\mathrm{PD}(n)$ space is a generalization of a uniformly contractible $n$-manifold. If $G$ is the fundamental group of a compact aspherical $n$-manifold $N$ (possibly with boundary), then by Davis's reflection group trick $G$ is a subgroup of the fundamental group of another closed, aspherical $n$-manifold $M$. Then $G$ acts properly on the universal cover $\widetilde M$, which is a coarse $\mathrm{PD}(n)$ space.
Hence, Theorem \ref{t:introthm} follows from Theorem \ref{t:main theorem}.

\begin{Example} Here is a specific example, taken from \cite{Leary-Minasyan}, where we can apply Theorem~\ref{t:main theorem}. 
Let $\zz^2 = \langle a,b \rangle$, $L'$ the subgroup generated by $a^2b^{-1}$ and $ab^2$, and $L''$ the subgroup generated by $a^2b$ and $a^{-1}b^2$. 
Then the transformation taking $L$ to $L'$ is rotation through $\arccos(3/5)$, in particular, the corresponding group is CAT(0) but the order of the corresponding matrix is infinite. By Theorem~\ref{t:main theorem}, the boundary of the group embeds into $S^3$, but the group does not contain any finite index subgroup that acts properly on a coarse PD($4$) space.
\end{Example}

 \subsection{On the proof of Theorem~\ref{t:main theorem}}
 Kapovich and Kleiner proved the following: 

\begin{Theorem}[\cite{KK}]\label{t:KK theorem}
    There is no finite index subgroup of $BS(p,q)$ that acts properly on a coarse $\mathrm{PD}(3)$ space if $|p|\neq |q|$.
 \end{Theorem}

 Let us briefly sketch the proof of the above theorem from~\cite[Section 9]{KK}. We start by recalling a few definitions. A map $f:X\to Y$ between two metric spaces is a \textbf{coarse embedding} if there are two proper functions $\rho_-,\rho_+:[0,\infty)\to [0,\infty)$ such that for all $x,y\in X$ we have
 \[
 \rho_-(d(x,y))\leq d(f(x),f(y))\leq \rho_+(d(x,y)).
 \]
 A map $f:X\to Y$ between two metric spaces is called a \textbf{coarse equivalence} if $f$ is a coarse embedding and there exists an $r\geq 0$ such that $Y$ is contained in the $r$-neighborhood of $f(X)$.
 For a given coarse embedding $f:X\to Y$, we say a map  $g:X\to X$ \textbf{coarsely extends} to a map
$\bar{g}:Y\to Y$ if there exists $r\geq 0$ such that $d(f\circ g(x),\bar{g}\circ f(x))\leq r$ for all $x\in X$.

 Suppose that there is a finite index subgroup of $BS(p,q)$ that acts properly on a coarse $\mathrm{PD}(3)$ space and $|p|\neq |q|$.
The universal cover of the Cayley 2-complex of $BS(p,q)$ is homeomorphic to $T\times \rr$, where $T$ is the Bass-Serre tree of the splitting, and hence there is a $BS(p,q)$-equivariant coarse equivalence between $BS(p,q)$ and $T\times \rr$.
It follows that there is a $BS(p,q)$-equivariant coarse embedding of $T\times \rr$ into a coarse $\mathrm{PD}(3)$ space.
Now, there are two key steps to get a contradiction from here.
    In the first step, one shows that given a coarse embedding of $T\times \rr$  into a coarse $\mathrm{PD}(3)$ space, there is a cyclic order (see Definition~\ref{d:cyclic order}) on $\partial_\infty T$ that is respected by any isometry of $T\times \rr$ that coarsely extends to a coarse equivalence of the coarse PD($3$) space.
    In the second step,  one shows that no finite index subgroup of $BS(p,q)$ respects the cyclic order on $\partial_\infty T$ obtained in the first step when $|p|\neq |q|$.
    This gives a contradiction.

We follow the same approach to prove Theorem~\ref{t:main theorem}.  
Since $G(A,L')$ acts geometrically on $T\times \ee^n$ by Theorem~\ref{t:infinite index},
there is a $G(A,L')$-equivariant coarse equivalence between $G(A,L')$ and $T\times \ee^n$.

 Theorem~\ref{t:main theorem} will follow from the following two propositions.

\begin{Prop}[Cyclic order]\label{t:Cyclic order}
Given a coarse embedding of $T\times \ee^n$ into a coarse $\mathrm{PD}(n+2)$ space $X$, there exists a cyclic order on $\partial_\infty T$ such that for any
     $g\in \Isom(T)\times \Isom(\ee^n)$ that coarsely extends to a coarse equivalence $X\to X$, the action of $g$ on the $T$ factor respects the cyclic order on $\partial_\infty T$.
\end{Prop}

The existence of this cyclic order was also shown in~\cite{KK}, though the details are omitted.
A complete proof was later given in~\cite{HST2} for the case where $T$ has finitely many ends.  
Our main contribution is a complete and alternative proof of the general case using the coarse cohomology framework of \cite{BO}. 
   
\begin{Prop}[Incompatible group action]\label{t:Incompatible group action}
Suppose $A$ has infinite order and is conjugate
in $GL(n,\rr)$ to an orthogonal matrix. 
If $\Gamma$ is a finite index subgroup of $G(A,L')$, then the action of $\Gamma$ on $\partial_\infty T$ does not respect any cyclic order on $\partial_\infty T$ as in Proposition~\ref{t:Cyclic order}.
\end{Prop}

Now we give a proof of Theorem~\ref{t:main theorem} assuming Proposition~\ref{t:Cyclic order} and Proposition~\ref{t:Incompatible group action}.

\begin{proof}[Proof of Theorem~\ref{t:main theorem}]
    On the contrary, suppose $G$ acts on a coarse $\mathrm{PD}(n+2)$ space $X$.
   Then we consider the orbit map $G\rightarrow X$, $g\mapsto gx_0$ for some $x_0\in G$.
   This is a $G$-equivariant coarse embedding from $G$ to $X$.
   Since $G$ acts geometrically on $T\times \ee^n$ by isometries~\cite[Theorem 7.5]{Leary-Minasyan}, there is a $G$-equivariant coarse equivalence between $G$ and $T\times \ee^n$.
   It follows that there is a coarse embedding of $T\times \ee^n$ into $X$.
   Furthermore, since $G$ acts on $X$, every element $g\in G<\Isom(T)\times \Isom(\ee^n)$ coarsely extends to a coarse equivalence $X$.
   By Proposition~\ref{t:Cyclic order}, there is a cyclic order on $\partial_\infty T$ that is respected by every element in $G$.
  By invoking Proposition~\ref{t:Incompatible group action}, we arrive at a contradiction.
\end{proof}

\subsection{Overview}
The remainder of the article is dedicated to proving Proposition~\ref{t:Cyclic order} and Proposition~\ref{t:Incompatible group action}. In Section~\ref{s:preliminaries}, we provide the necessary background for these proofs. Section~\ref{s:cyclic order} contains the proof of Proposition~\ref{t:Cyclic order} (Proposition~\ref{p:cyclic order high dimension}), while Section~\ref{s:incompatible group action} is devoted to the proof of Proposition~\ref{t:Incompatible group action}~(Proposition~\ref{p:incompatible action}).

\subsection{Acknowledgements}
We thank Boris Okun for useful comments on an earlier draft of this article. The second author is supported by NSF grants DMS-2203325 and DMS-2505290.

\section{Preliminaries}\label{s:preliminaries}
In this section, we review the concept of coarse (co)homology of the complement, as developed in~\cite{BO}, which serves as the main technical tool used throughout this article.

We first fix some notation.
 Let $(X, d)$ be a metric space.
For $A \subset X$ denote $$N_{R}(A)=\{x \in X \mid d(x, A) <R \}.$$
We say $A$ is \emph{coarsely contained} in $B$, denoted by $A \csubset B$, if $A \subset N_R(B)$ for some $R$.
Two subsets are \emph{coarsely equal}, $A \ceq B$, if $A \csubset B$ and $B \csubset A$, and $A \ccap B \ceq C$ if for all sufficiently large $R$, $N_{R}(A) \cap N_{R}(B) \ceq C$.
The coarse intersection is not always well-defined, it may happen that the coarse type of $N_{R}(A) \cap N_{R}(B) $ does not stabilize as $R$ goes to infinity.
However the notion ``coarse intersection is coarsely contained in'' is well-defined;
$A \ccap B \csubset C$ means that for any $R$, $N_{R}(A) \cap N_{R}(B) \csubset C$.

\subsection{Coarse (co)homology} 
We will refer to points in $X^{n+1}$ as $n$-simplices.
In what follows, we will need to measure distances between simplices of different dimensions. 
A convenient way to do this is to stabilize simplices by repeating the last coordinate, as follows.
Denote by $X^\infty$ the subset of the product of countably many copies of $X$, consisting of eventually constant sequences.
Equip $X^{\infty}$ with the $\sup$ metric.
Let $i:X^{n+1} \to X^\infty$ denote the map $ (x_{0}, \dots, x_{n}) \mapsto (x_{0}, \dots, x_{n}, x_{n}, x_{n}, \dots)$.
For a function $\phi:X^{n+1} \to \zz$ define its stabilized support
\[
\Supp{\phi} = \{i(\gs) \mid \gs \in X^{n+1} \text{ and } \phi(\gs) \neq 0\} \subset X^{\infty}.
\]
Let $\gD=i(X)$ denote the diagonal of $X^{\infty}$.

We now define coarse (co)homology theories, following Roe~\cite[Section 2.2]{r93} using the language of \cite{BO}. Let $G$ be an abelian group. First consider the following cochain complex 
\[
\C[*](X)=\{\phi: X^{*+1} \to G\}
\]
with the coboundary operator
\[
d(\phi)(x_0, \dots, x_n)=\sum_{i=0}^n(-1)^{i} \phi(x_0, \dots, \hat{x}_i, \dots, x_n).
\]
This is an acyclic complex. 
The coarse cochain complex is
\[
\Cx(X;G) = \{\phi \in \C(X;G) \mid \Supp{\phi} \ccap \gD \ceq * \}.
\]
The coboundary operator $d$ maps $\Cx(X)\to \Cx[*+1](X)$.
The coarse cohomology $\Hx(X)$ is defined to be the cohomology of the complex $(\Cx(X),d)$.

There is similarly a dual coarse homology theory $\hx$, which is due to
Block--Weinberger~\cite[Section 2]{blockweinberger} and Higson--Roe~\cite[Section 2]{higsonroe}.
However, we are going to follow Hair~\cite[Section 1.6.1]{h10} which gives a more appropriate version of coarse homology for our needs.

We consider the complex of finitely supported integral chains
\[
\cf(X;G):=\{\sum_{i=1}^k c_i \sigma_i \mid c_i \in G, \, \sigma_i\in X^{*+1}\}
\]
equipped with the usual boundary map, defined on the basis by
\[
\partial (x_0, \dots, x_n):=\sum_{i=0}^n (-1)^i(x_0, \dots, \hat{x}_i, \dots, x_n).
\]
Note that the boundary map $\partial$ is well-defined on a larger complex of \emph{locally finite} chains $\clf(X; G)$ which consists of chains $c$ satisfying that for any bounded $B \subset X$ only finitely many simplices in $c$ have vertices in $B$.
The \emph{coarse homology} $\hx(X)$ is the homology of the following subcomplex of $\clf(X)$
\[
\cx(X;G):=\{c\in \clf(X;G)\mid \Supp{c}\csubset \gD\}
\]
equipped with the boundary operator $\partial$. 
\begin{Example}
    Roe showed that for uniformly contractible proper metric spaces the coarse cohomology is isomorphic to the compactly supported Alexander--Spanier cohomology~\cite[Proposition 3.33]{r93}.
    In particular, this applies to the universal cover of finite aspherical complexes.
    In this case the coarse homology is isomorphic to the locally finite homology~\cite[Chapter 2]{r96}.
    For example,
    \[
    \hx(\rr^n;G)=\Hx(\rr^n;G)=
    \begin{cases}
        G & *=n, \\
        0 & \text{otherwise}.
    \end{cases}
    \]
\end{Example}

For a subset $A\subset X$, we denote the set $i(A)\subset \gD$ by $\gD_A$.
The coarse cohomology of the complement of $A$ is $X$, denoted by $\Hx(X-A;G)$,  is the cohomology of the following complex.

  \[ \Cx[n](X-A;G)= \{ \phi \in \C[n](X;G) \mid \Supp{\phi} \ccap \gD \csubset \gD_A \}
  \]

  \subsection{Coarse complementary component}
  A subset $C\subset X$ is a \emph{coarse complementary component} of $A$ if $C\ccap (X-C)\csubset A$.
  A coarse complementary component $C$ of $A$ is \emph{shallow} if $C\csubset A$, otherwise it is \emph{deep}.
The \emph{$r$-boundary} of $C$ is
\[
\partial_r C := \{x \in X-C \mid d(x, C) \leq r\}.
\]
\begin{Lemma}[\cite{BO}]\label{l:complement criterion}
    $C$ is a coarse complementary component of $A$ if and only if $\partial_r(C) \csubset A$ for all $r$.
\end{Lemma}

The complementation map $C \to X-C$ is a free involution on $2^{X}$.
Let $\cS$ be the quotient space, we will think of its elements as unordered pairs $\{C, X-C\}$, and refer to them as \emph{separations} of $X$.
Alternatively, since $X-C=X \symdiff C$, one can think of $\cS$ as a quotient of abelian groups $2^{X}/ \{\emptyset, X\} $, where the addition is given by the symmetric difference.

Let $\cC_{A}$ denote the collection of all coarse complementary components of $A$.
As before, each $C \in \cC_{A}$ determines a separation $\{C, X-C\} \in \cS$, we will refer to it as a \emph{separation of $X$ with respect to $A$}.
Such a separation is \emph{deep} if both $C$ and $X-C$ are deep, and \emph{shallow} otherwise.
We will say that \emph{$A$ separates $X$} if there exists a deep separation of $X$ with respect to $A$.
Let $\cS_A$ denote the collection of all such separations, and let $\cS\cS_A$ be the subcollection of shallow separations.
It is shown in ~\cite[Section 5]{BO} that $\cS\cS_{A}$ and  $\cS_{A}$ are subgroups of $\cS$. Let $\cD\cS_A=\cS_A/\cS\cS_A$, its nonzero elements are equivalence classes of deep separations of $X$ with respect to $A$ modulo shallow ones. 

Suppose $\zz/2=\langle \mathbf{1} \rangle$. The next two lemmas from~\cite{BO} relates coarse complementary components of $A$ to $\Hx[1](X-A;\zz/2)$.
\begin{Lemma}[\cite{BO}]\label{l:coarse complement}
The map $(C,X- C)\mapsto d(\mathbf{1}_C)$ induces an isomorphism $\cD\cS_A \to \Hx[1](X-A; \zz/2)$.
\end{Lemma}
\begin{Lemma}[\cite{BO}]\label{l:dim}
    Suppose $\{C_{\alpha}\}$ is a pairwise disjoint collection of deep coarse complementary components of $X$ with respect to $A$.
    Then any proper subcollection of $\{C_{\alpha}\}$ maps to a linearly independent subset of $\Hx[1](X-A; \zz/2)$.
\end{Lemma}

Coarse complementary components are particularly easy to study if the underlying space $X$ is a geodesic space. 
For example, in the geodesic setting, coarse complementary components can be identified with a union of path components~\cite[Proposition 5.7]{BO}.
For our purpose, it suffices to assume $X$ is geodesic in a coarse sense.
This motivates the following definitions.

\begin{Definition}
An \emph{$s$-path} between $x$ and $y$ is a finite sequence of points $\{x=x_0,\ldots, x_n=y\}$ so that $d(x_i,x_{i+1})\leq s$. A metric space $X$ is \emph{$s$-path connected} if there is an $s$-path between any two points in $X$. 
An \emph{$s$-path component} of $X$ is a maximal subset of $X$ that is $s$-path connected.
 A metric space $X$ is \emph{uniformly $s$-path connected} if there is a function $\rho:[0,\infty)\to [0,\infty)$ such that if $d(x,y)\leq r$, then there exists an $s$-path of diameter at most $\rho(r)$ between $x$ and $y$.
We say $X$ is \emph{uniformly coarse-path connected} if $X$ is uniformly $s$-path connected for some $s$.
\end{Definition}
\begin{Example}
    We observe that any space that is coarsely equivalent to a uniformly 1-acyclic space is uniformly coarse-path connected. In particular, any uniformly contractible space is uniformly coarse-path connected.
    \end{Example}

The next two lemmas relate $s$-path connectedness to coarse complementary components. 
\begin{Lemma}\label{p:uniform}
    If $X$ is a uniformly $s$-path connected space and $A\subset X$, then for any $r$, the collection of $s$-path components of $X-N_r(A)$ is a  collection of uniform coarse complementary components. 
    In particular, the union of any subcollection of $s$-path components of $X-N_r(A)$ is a coarse complementary component.
\end{Lemma}
\begin{proof}
    Fix $r$, and let $C$ be an $s$-path component of $X-N_r(A)$.
    Let $\rho:[0,\infty)\to [0,\infty)$ be the function associated to the uniform $s$-connectedness of $X$.
    By Lemma~\ref{l:complement criterion}, it is enough to show that for any $R$, $\partial_R(C) \subset N_{\rho(R)+r} A$.

    Let $x\in \partial_R(C)$.
    Moreover, assume that $x\notin N_r(A)$, otherwise there is nothing to prove.
    Then there exists $y \in C$ such that $d(x, y)\leq R$.
    Since $x \notin C$ and $y \in C$, $x$ and $y$ are in different $s$-path components of $X-N_r(A)$, thus an $s$-path $\{y=x_0,x_1,\ldots,x_n=x\}$ between  $y$ and $x$ must dip into $N_{r}(A)$. Therefore, there exists $x_i$ on that $s$-path such that $d(x_i, A) < r$.
    Since $X$ is uniformly $s$-connected, we can assume $d(x,x_i)\leq \rho(R)$.
    Then by triangle inequality $d(x, A) \leq d(x, x_i) + d(x_i, A) < \rho(R)+r$. 
    Hence $x\in N_{\rho(R)+r}(A)$. The claim follows.
\end{proof}

\begin{Lemma}\label{l:connected and complementary comp}
    Let $C$ be a coarse complementary component of $A$ in $X$. For each $s$, there exists $p:=p(s,C)$ such that if $M\subset X-N_p(A)$ is  $s$-path connected then
     $M\subset C$ or $M\subset X-C$.
\end{Lemma}
\begin{proof}
Since $C$ is a coarse complementary component of $A$, there exists $s'$ such that $\partial_s(C)\subset N_{s'}(A)$ by Lemma~\ref{l:complement criterion}.

Let $p:=s'+s$. If $M\subset X-C$, then we are done.
Otherwise, suppose $M\cap C\neq \emptyset$. We claim that $M\subset C$.

Pick $x\in M\cap C$. Let $y\in M$. We want to show $y\in C$.
Since $M$ is $s$-path connected,
 we can choose points $\{x=x_0, \dots, x_n=y\}$ in $M$ so that $d(x_i, x_{i+1}) \leq s$ for all $i$.
 Since $M\subset X-N_p(A)$, $d(x_i,A)>p$ for all $i$.
If $x_i\in C$, then $d(x_{i+1}, A) \geq d(x_i, A)-d(x_i, x_{i+1}) > p-s \geq s'$, so $x_{i+1} \notin N_{s'}(A)$, and therefore $x_{i+1} \notin \partial_s(C)$.
    On the other hand, $d(x_{i+1}, C) \leq d(x_i, x_{i+1}) \leq  s$.
    Combining these together, we get $x_{i+1}\in C$.
    So by induction, $y\in C$.
\end{proof}

\subsection{Coarse $\mathrm{PD}(n)$ spaces}
Coarse $\mathrm{PD}(n)$ spaces are the coarse version of manifolds, in the sense that they admit a coarse version of Poincare duality. 
Let us now recall the definition of a coarse $\mathrm{PD}(n)$ space from~\cite{BO}.

\begin{Definition}\label{def pdn}
    A metric space $X$ is a \emph{coarse $\mathrm{PD}(n)$ space}, if there exist chain maps $p: \C(X;\zz) \to \cx[n-*](X;\zz)$ and $q: \cx[n-*](X;\zz) \to \C(X;\zz)$, so that $pq$ and $qp$ are chain homotopic to identities via chain homotopies $G:\cx(X;\zz) \to \cx[*+1](X;\zz)$ and $F: \C(X;\zz) \to \C[*-1](X;\zz)$ which are controlled:
    \begin{align*}
        \forall \phi\in \C(X;\zz) &\qquad \Supp{ p(\phi)} \csubset \Supp{\phi}, \\
        \forall \phi\in \C(X;\zz) &\quad \Supp{ F(\phi)} \ccap \gD \csubset \Supp{\phi} , \\
        \forall c \in \cx(X;\zz) &\quad \Supp{q(c)} \ccap \gD \csubset \Supp{c}, \\
        \forall c\in \cx(X;\zz) &\qquad \Supp{ G(c)} \csubset \Supp{c}.
    \end{align*}
\end{Definition} 

\begin{Example}
    Any proper, uniformly acyclic $n$-manifold is a coarse $\mathrm{PD}(n)$ space~\cite[Corollary 8.3]{BO}. In particular, the universal cover of a closed, aspherical $n$-manifold is a coarse $\mathrm{PD}(n)$ space. 
\end{Example}
Let $\delta(X)$ be the diagonal subset of $X\times X$, and define $\C(X\otimes X):=\Hom{(\cf(X)\otimes \cf(X);\zz)}$. The cohomology $\Hx[n](X\otimes X-\delta(X))$ is the cohomology of the following cochain complex
\[
\Cx(X\otimes X-\gd(X)):=\{\phi\in \C(X\otimes X)\mid \Supp{\phi}\ccap \gD_{X\times X} \csubset \gD_{\gd(X)}\}.
\]
 This complex is chain homotopy equivalent to $\Cx(X\times X-\gd(X))$ via coarsely support preserving homotopies (see~\cite[Section 7]{BO}). One advantage of working with  $\C(X\otimes X)$ rather than $\C(X\times X)$,  is that it offers a more convenient setting for defining slant products.

 Suppose $\phi\in \C[n](X\otimes X)$ and $\gs^k\in \cf[k](X)$ is a $k$-simplex.
The slant product $\phi/\gs\in \C[n-k](X)$ is defined by
\[
(\phi/\gs^k)(\gt^{n-k}) = \phi(\gs \otimes \gt).
\]

  We now recall the coarse Alexander duality theorem from~\cite[Theorem 6.2]{BO}.
\begin{Theorem}[Coarse Alexander duality]\label{CAD}
 Suppose $X$ is a coarse $\mathrm{PD}(n)$ space. Then for any $A\subset X$ and for any finitely generated abelian group $G$, we have
 \[
 \Hx[k](X-A;G)\cong \hx[n-k](A;G).
 \]

Furthermore, there exists a class $U\in \Hx[n](X\otimes X-\delta(X))$ such that the slant product with $U$ gives the above isomorphism from $\hx[n-k](A)$ to $\Hx[k](X-A)$.
\end{Theorem}

The following is a corollary of the above theorem~\cite[Corollary 6.3]{BO}.
\begin{Corollary}\label{c:pdn}
    Suppose $X$ is a coarse $\mathrm{PD}(n)$ space, and $G$ be a finitely generated abelian group.
    Then:
        \item\label{i:hx of PD(n)}

        $$ \Hx[*](X;G)=\hx[*](X; G)=
        \begin{cases}
            G & *=n,\\
            0 & \text{otherwise}.
        \end{cases}$$
        
\end{Corollary}

As mentioned before, any proper, uniformly acyclic $n$-manifold is a coarse $\mathrm{PD}(n)$ space.
The next lemma can be thought as a partial converse to this.

\begin{Lemma}\label{pdn is coarse connected}
    If $X$ is a coarse $\mathrm{PD}(n)$ space, then $X$ is uniformly coarse-path connected. 
\end{Lemma}
\begin{proof}
Since $X$ is a coarse $\mathrm{PD}(n)$ space, by definition there exist chain maps $p: \C(X;\zz) \to \cx[n-*](X;\zz)$ and $q: \cx[n-*](X;\zz) \to \C(X;\zz)$, so that $pq$ is chain homotopic to identity via a chain homotopy $G:\cx(X;\zz) \to \cx[*+1](X;\zz)$ which are controlled.

Let $\sigma$ be a $1$-simplex (recall that an $1$-simplex is a pair of points). 
We note that
\begin{align*}
    \partial [pq(\sigma)-G(\partial \sigma)]=\partial [\sigma+\partial G(\sigma)]=\partial \sigma
\end{align*}
In other words, $pq(\sigma)-G(\partial \sigma)$ gives a 1-chain that `connects' the two vertices of $\sigma$. 
Moreover, if $|pq(\sigma)-G(\partial \sigma)|\subset N_s(\Delta)$, then all the 1-simplices in the support of $pq(\sigma)-G(\partial \sigma)$ have diameter $\leq s$.
If $|pq(\sigma)-G(\partial \sigma)|\subset N_r(|\gs|)$ for some $r$, then the support of the $1$-chain $pq(\sigma)-G(\partial \sigma)$ remains within the $r$-neighborhood of its vertices of $\sigma$. In particular, $|pq(\sigma)-G(\partial \sigma)|$ will be bounded in this case, and since $pq(\sigma)-G(\partial \sigma)\in \cx[1](X)$, it will follow that $pq(\sigma)-G(\partial \sigma)$ is a finite $1$-chain.

Therefore, it is enough to prove the following two properties of $pq(\sigma)-G(\partial \sigma)$: Firstly, there exists an $s$ such that $|pq (\sigma)-G(\partial \sigma)|\subset N_s(\Delta)$ for all 1-simplices  $\sigma$. Secondly, there exists a function $\rho:\rr\to \rr$ such that $|pq(\sigma)-G(\partial \sigma)|\subset N_{\rho(\diam(\sigma))}(\sigma)$ for all 1-simplices $\sigma$, where $\diam(\sigma)$ is the distance between two vertices in $\sigma$.

It suffices to prove the properties for $pq(\sigma)$ and $G(\partial \sigma)$ separately.
Both properties can be proved by choosing good representatives of $p,q$ and $G$.
By (the proof of)~\cite[Theorem 8.2]{BO}, there exist $c\in \cx[n](X)$ and $U\in \Cx[n](X\otimes X-\delta(X))$, such that $p$ is cap product with $c$ and $q$ is a slant product with $U$ (up to sign).
It follows that $pq(\sigma)$ satisfies both the properties.
It remains to show that $G(\partial \sigma)$ satisfies the properties.

To this end, we recall from the proof of \cite[Theorem 8.2]{BO} that the homotopy $G$ is composition of two maps.
The first one is the projection map $p_{1_*}:\cf(X\otimes X)\to \cf(X)$, $\tau_1\otimes \tau_2\to \tau_1$.
The second one is the homotopy $H$ between $\sigma\mapsto U \capp (c\otimes \sigma)$ and $\sigma\mapsto T^*U \capp (\sigma\otimes c)$ where $T$ is the involution map $\C(X\otimes X)\to \C(X\otimes X)$ as in~\cite[Lemma 7.4(3)]{BO}.
From the proof of~\cite[Lemma 7.4(3)]{BO}, it follows that $H$ is a homotopy between $\sigma\to T_*(U\capp (c\otimes \sigma))$ and $\sigma \to U\capp (c\otimes \sigma))$, where $T_*:\cf(X\otimes X)\to \cf(X\otimes X)$ is the involution map $\tau_1\otimes \tau_2\to \tau_2\otimes \tau_1$ up to sign.
It follows that there exists a function $\rho:\rr\to \rr$ such that $|H(\sigma)|\subset N_{\rho(\diam(\sigma))}(\Delta)$ and $|H(\sigma)|\subset N_{\rho(\diam(\sigma))}(\sigma)$ for all $\sigma$.
Consequently there exists $s$ such that $|H(\sigma)|\subset N_{s}(\Delta)$ for all $0$-simplices.
Since $G=p_{1_*}H$,  the same properties hold for $G$.
\end{proof}
\subsection{Separating $\mathrm{PD}(n)$ spaces by $\mathrm{PD}(n{-}1)$ spaces}
The following lemma collects several important properties of coarse complementary components of a coarse $\mathrm{PD}(n{-}1)$ space coarsely embedded in a coarse $\mathrm{PD}(n)$ space.

\begin{Lemma}\label{l:coarse complement for PDn}
   Let $X$ be a coarse $\mathrm{PD}(n)$ space. Let $A$ be a coarse $\mathrm{PD}(n-1)$ space which is coarsely embedded in $X$.
   Then
   \begin{enumerate}
       \item  $\Hx[1](X-A;\zz)=\zz$. Furthermore, there exists a coarse complementary component $C$ of $A$ in $X$ such that the class of $d(1_C)$ generates $\Hx[1](X-A;\zz)$. 
\item If $C$ is a deep coarse complementary component of $A$ in $X$, then $A\csubset C$.
       
       \item If $d(1_M)$ and $d(1_N)$ represent the same nontrivial class in $\Hx[1](X-A;\zz)$, then $M\ceq N$.
        \item Suppose  that $(M,X-M)$ is a deep separation of $X$ with respect to $A$. 
Then, for any $r$ the union of shallow $s$-path components of $M-N_r(A)$ is shallow: they all are contained in $N_{R}(A)$ for some $R$.
        Moreover, for all $r$, $M-N_r(A)$ contains exactly one deep $s$-path connected component.
         
   \end{enumerate}

   \begin{proof}
       \begin{enumerate}
           \item \noindent By Theorem~\ref{CAD}, we have $\Hx[1](X-A;\zz)=\hx[n-1](A;\zz)$ and by Corollary~\ref{c:pdn}, we have $\hx[n-1](A;\zz)=\zz$. 
           Consequently, $\Hx[1](X-A;\zz)=\zz$.
           With $\zz/2$-coefficients, we similarly have $\Hx[1](X-A;\zz/2)=\zz/2$. 
           Let $q:\zz\to \zz/2$ be the surjective map.
            Lemma~\ref{l:coarse complement} implies that there exists a deep coarse complementary component $C$ of $A$ in $X$ such that $[d(q\circ 1_C)]$ generates $\Hx[1](X-A;\zz/2)$.
           Note that $q$ induces a map $\Hx[1](X-A;\zz)\to \Hx[1](X-A;\zz/2)$ and under this map $[d(1_C)]$ goes to $d(q\circ 1_C)$. It follows that $d(1_C)$ is nontrivial element in $\Hx[1](X-A;\zz)$. 
           We claim that $[d(1_C)]$ must be a generator of $\Hx[1](X-A;\zz)$.
           First observe that $C\neq X$, otherwise $d(1_C)$ would be trivial. 
           Since the complex $\C(X)$ is acyclic, there exists $f\in \C[0](X)$ such that $[d(f)]$ generates $\Hx[1](X-A;\zz)$. 
           Therefore, $[d(1_C)]=m\cdot[d(f)]$ for some non-zero integer $m$.
           This means there exists $g\in \Cx[0](X-A;\zz)$ such that 
          $1_C-mf-g$ is some constant function $x\mapsto k$ for some $k\in \zz$. 
          Since $g\in \Cx[0](X-A;\zz)$, by definition $|g|\csubset A$.
           Since $|g|\csubset A$ 
           and both $C$ and $X- C$ are not coarsely contained in $A$, we can find $x\in C$ and $y\in X- C$ such that $g(x)=g(y)=0$.
           Applying $1_C-mf-g$ on $x$ and $y$, we obtain that $1-mf(x)=k=mf(y)$.
           In other words, both $k$ and $k-1$ must be divisible by $m$. This implies $|m|=1$. The claim follows.
           \\

           \item  
This is proved in~\cite[Corollary 6.5(1)]{BO}.
\\           
\item First observe that $d(1_M)\in \Cx[1](X-A;\zz)$ implies that $M$ is a coarse complementary component. 
Since $[d(1_M)]$ is a nontrivial class, it follows from the proof of the first part of the lemma that $d(1_M)$ represents the generator of $\Hx[1](X-A,\zz)$.
Under the surjective map $q:\zz\to \zz/2$, $d(q\circ 1_M)$ represents the non-trivial class in $\Hx[1](X-A;\zz/2)$.
Since $\Hx[1](X-A;\zz/2)=\zz/2$,
by Lemma~\ref{l:coarse complement}, we have $\cD\cS_{A}=\zz/2$.
In particular, $(M,X-M)$ represents the unique nontrivial element in $\cD\cS_{A}$.
The same is true for $(N,X-N)$.
So, $(M,X-M)$ and $(N,X-N)$ represent the same class in $\cD\cS_{A}$.
Either, $M\symdiff N\csubset A$ or $M\symdiff (X-N)\csubset A$.

Suppose, $M\symdiff N\csubset A$. By (2), we have $A\csubset M$ and $A\csubset N$. It follows that $M\ceq N$. Now, suppose, $M\symdiff (X-N)\csubset A$.
 Then $|1_M-1_{X-N}|\csubset A$, and hence $1_M-1_{X-N}\in \Cx[0](X-A;\zz)$.
 In other words, $[d(1_M-1_{X-N})]$ represents the trivial class and therefore $[d(1_M)]=[d(1_{X-N})]$ .
Since, $d(1_{X-N})=-d(1_N)$, we obtain $[d(1_M)]=-[d(1_N)]$ which is a contradiction.
 \\
 \item
Let $M$ be a coarse complementary component of $A$ in $X$.
If, for some $r$, the union of shallow path components is not shallow, then there is a sequence $\{C_i\}$ of shallow path components of $X-N_r(A)$ such that $C_i$ is not contained in $N_i(A)$ for all $i\in \nn$.
By Lemma~\ref{pdn is coarse connected}, $X$ is uniformly coarse-path connected.
    Therefore the union of any subsequence is a deep coarse complementary component by Lemma~\ref{p:uniform}, and since there are infinitely many subsequences whose unions are pairwise disjoint, $\dim_{\zz/2} \Hx[1](X-A; \zz/2)$ is infinite by Lemma~\ref{l:dim}.

Since $M$ is a deep coarse complementary component, it follows that $M-N_r(A)$ has at least one deep $s$-path component for any $r$.
 By Lemma~\ref{l:dim} for any $r$, there exists at most one deep $s$-path components in $M-N_r(A)$.
 The claim follows.
\end{enumerate}
   \end{proof}
\end{Lemma}

\section{Cyclic order}\label{s:cyclic order}

In this section, we prove Proposition~\ref{t:Cyclic order}. We start with the formal definition of cyclic order and order respecting maps.

\begin{Definition}[Cyclic order]\label{d:cyclic order}
     A cyclic order on a set $X$ is a relation $\mathcal{C}\subset X^3$, written $[a,b,c]$, that satisfies the following axioms:
  \begin{itemize}
 \item Cyclicity: If $[a, b, c]\in \mathcal{C}$ then $[b, c, a]\in \mathcal{C}$.
\item Asymmetry: If $[a, b, c]\in \mathcal{C}$ then $[c, b, a]\notin \mathcal{C}$.
\item Transitivity: If $[a, b, c]\in \mathcal{C}$ and $[a, c, d]\in \mathcal{C}$ then $[a, b, d]\in \mathcal{C}$.
\item Connectedness: If $a,b,$ and $c$ are distinct, then either $[a,b,c]\in \mathcal{C}$ or $[a,c,b]\in \mathcal{C}$.
\end{itemize}
Intuitively, $[a, b, c]\in \mathcal{C}$ means ``after a, one reaches b before c".
\end{Definition} 
\begin{Definition}[Order respecting maps]
   Suppose $X$ is a set with a cyclic order $\mathcal C$. A map $f:X\rightarrow X$ is said to preserve the cyclic order if for all $a,b,c\in X$,
   \[[a,b,c]\in \mathcal C\implies [f(a),f(b),f(c)]\in \mathcal C.
   \]
   A map $f:X\to X$ is said to reverse the cyclic order if for all $a,b,c \in X$
   \[
   [a,b,c]\in \mathcal C\implies [f(c),f(b),f(a)]\in \mathcal C.
   \]
   A map $f:X\rightarrow X$ is said to respect the cyclic order $\mathcal C$ if $f$ either preserves or reverses the cyclic order of $X$.
\end{Definition}

\begin{Example}\label{e:cycle}
The circle $S^1$ has two canonical cyclic orders: clockwise and counter-clockwise. 
$[a,b,c]$ is in the clockwise cyclic order of $S^1$ iff in the clockwise direction after $a$, one reaches $b$ before $c$. 
Similarly, one can define the counter-clockwise cyclic order.
Furthermore, any orientation preserving (respectively reversing) homeomorphism $S^1\to S^1$ preserves (respectively reverses) the cyclic order.
\end{Example}

Before delving into the main proof, let us roughly describe the cyclic order of Proposition~\ref{t:Cyclic order} in the simplest case: 
when we have a quasi-isometric embedding from $T$ to $X:=\mathbb H^2$. 
This coarse embedding induces an embedding from $\partial_\infty T$ to $\partial_\infty \mathbb H^2=S^1$. 
 Picking a cyclic order on $S^1$ will then induce a cyclic order on $\partial_\infty T$.

In general, a coarse $\mathrm{PD}(n)$ space $X$ might not have a well defined boundary and even when the boundary exists, a coarse embedding from $T\times \ee^n$ to $X$ may not induce a well defined map between their boundaries. 

The key observation is that any coarse embedding of $T\times \ee^n$ into $X$ gives a cyclic order on the set of coarse complimentary components.
By coarse Alexander duality, these coarse complimentary components are in one-to-one correspondence with the set of pair of points of $\partial_\infty T$.
This correspondence induces a cyclic order on $\d_\infty T$.

We now proceed to the proof of the general case. First we recall  Proposition~\ref{t:Cyclic order} from the introduction.

\begin{Prop}\label{p:cyclic order high dimension}

 Given a coarse embedding of $Y:=T\times \ee^n$ into a coarse $\mathrm{PD}(n+2)$ space $X$, there exists a cyclic order on $\partial_\infty T$ such that for any
     $g\in \Isom(T)\times \Isom(\ee^n)$ that coarsely extends to a coarse equivalence $X\to X$, the action of $g$ on the $T$ factor respects the cyclic order on $\partial_\infty T$.
\end{Prop}

We divide the proof into two parts. 
In the first part we construct the cyclic order. 
In the second part, we show that the cyclic order is respected by certain coarse equivalences.   
    For the rest of the article, the omitted coefficients for (co)homology groups will be $\mathbb Z$.
    A bi-infinite geodesic in the tree $T$ with end points $a,b\in \partial_\infty T$, will be denoted by $(a,b)$. An infinite geodesic ray with endpoints $a\in T$ and $b\in \partial_\infty T$ will be denoted by $[a,b)$.

\subsection{The cyclic order}

    Let $f:Y\rightarrow X$ be the coarse embedding.
By Theorem~\ref{CAD}, there exists a coarse Alexander duality map
\[
\phi:\hx[n+1](f(Y))\rightarrow \Hx[1](X-f(Y)).
\]
 Composing with the following compositions of isomorphisms
 \[
\rh[0](\partial_\infty T)\rightarrow  \rh[n](\partial_\infty T * \mathbb{S}^{n-1})\rightarrow \hx[n+1](T\times \ee^n)\to \hx[n+1](f(Y)) \]
we can regard $\phi$ as an isomorphism from $\rh[0](\partial_\infty T)$ to $\Hx[1](X-f(Y))$.
Let $a,c\in \partial_\infty T$. 
Then $\phi(c-a)\in \Hx[1](X-f(Y))$.
Moreover, the map $\phi$ restricted to the class $c-a\in \rh[0](\partial_\infty T)$ factors through the inclusion map $\Hx(X-f((a,c)\times \ee^n))\rightarrow \Hx(X-f(Y))$.
Therefore, we can regard $\phi(c-a)$ as a class from $\Hx[1](X- f((a,c)\times \ee^n))$.
Since $f((a,c)\times \ee^n)$ is a coarse $\mathrm{PD}(n+1)$ space and $X$ is a coarse $\mathrm{PD}(n+2)$ space, by Lemma~\ref{l:coarse complement for PDn}(1), there exists a coarse complementary component, say $C(c-a)$, of $f((a,c)\times \ee^n)$ in $X$ such that $d(1_{C(c-a)})$ represents $\phi(c-a)$.

 Note that $C(c-a)$, as a subset of $X$, is not unique. However, by Lemma~\ref{l:coarse complement for PDn}(3), $C(c-a)$ is unique up to coarse equivalence.
 By $C(c-a)$, we will mean some set from the class it represents. 

\begin{Definition}

Fix a base point $o\in T$.
For any $a,b,c\in \partial_\infty T$, we say $$[a,b,c]\in \mathcal{C} \text{ if and only if } f([o,b)\times \ee^n)\csubset C(c-a).$$
\end{Definition}

We now show that $\mathcal{C}$ is a cyclic order.
\begin{Lemma}\label{l:symmconn}
    $\mathcal{C}$ is asymmetric and connected.
\end{Lemma}
\begin{proof}
For asymmetry, suppose $[a,b,c]$ and $[c,b,a]$ both are in $\mathcal{C}$. 
Then we have $f([o,b)\times \ee^n)\csubset C(c-a)$ and $f([o,b)\times \ee^n)\csubset C(a-c)$. 
Together, these imply $f([o,b)\times \ee^n)\csubset C(c-a)\ccap C(a-c)$.
Since $C(a-c)\ceq X- C(c-a)$ by Lemma~\ref{l:alt crit}(1) and $C(a-c)$ is a coarse complementary component of $f((a,c)\times \ee^n)$, it follows that $C(c-a)\ccap C(a-c)\csubset f((a,c)\times \ee^n)$.
 Therefore, $f([o,b)\times \ee^n)\csubset f((a,c)\times \ee^n)$. This is a contradiction since $f$ is a coarse embedding and  $[o,b)$ diverges from $(a,c)$ in $T$.

For connectedness, suppose that $[a,b,c]\notin \mathcal{C}$.
Let $A=f((a,c)\times \ee^n)$.
Since $f([o,b)\times \ee^n)$ is not coarsely contained in $C(c-a)$,  $f([o',b)\times \ee^n)$ is not a subset $C(c-a)$ for any $o'\in{} [o,b)$. 
For any $o'\in{} [0,b)$, since $[o',b)\times \ee^n$ is a path connected space and $f$ is a coarse equivalence, $f([o',b)\times \ee^n)$ is $s$-connected for some $s$.
We apply Lemma~\ref{l:connected and complementary comp} to conclude that there exists $p$ such that $f([o',b)\times \ee^n)\subset X-C(c-a)$ if $o'\in{} [o,b)$ and $d(o,o')>p$.
It follows that $f([o,b)\times \ee^n)\csubset X-C(c-a)$.
Hence $[c,b,a]\in \mathcal{C}$.
\end{proof}
In order to prove cyclicity and transitivity of  $\mathcal{C}$, we need the following lemma.
\begin{Lemma}\label{l:alt crit}
\begin{enumerate}
\item \label{1} $C(b-a)\ceq X- C(a-b)$.
    \item \label{2} $[a,b,c]\in \mathcal{C}$ iff $C(b-a)\csubset C(c-a)$ iff $C(c-b)\csubset C(c-a)$.
\end{enumerate}
    
\end{Lemma}

We first prove cyclicity and transitivity of  $\mathcal{C}$ assuming Lemma~\ref{l:alt crit} and later give a proof of Lemma~\ref{l:alt crit}.

To prove cyclicity, suppose $[a,b,c]\in \mathcal{C}$. Then by Lemma~\ref{l:alt crit}(2), $C(b-a)\csubset C(c-a)$. By Lemma~\ref{l:alt crit}(1), we have that $C(a-c)\csubset C(a-b)$.
Lemma~\ref{l:alt crit}(2) implies $[b,c,a]\in \mathcal{C}$.

 To prove transitivity, suppose $[a,b,c],[a,c,d]\in \mathcal{C}$.
 By Lemma~\ref{l:alt crit}(2), we have $C(b-a)\csubset C(c-a)$ and $C(c-a)\csubset C(d-a)$.
 Therefore, $C(b-a)\csubset C(d-a)$.
 Applying Lemma~\ref{l:alt crit}(2), we obtain $[a,b,d]\in \mathcal{C}$.

Therefore, $\mathcal{C}$ is a cyclic order. It remains to prove Lemma~\ref{l:alt crit}.

\begin{proof}[Proof of Lemma~\ref{l:alt crit}]
     To prove~\eqref{1}, first observe that $-d(1_{C(a-b)})=d(1_{X- C(a-b)})$. Moreover, $\phi(b-a)=-\phi(a-b)$ implies that $[d(1_{C(b-a)})]=-[d (1_{C(a-b)}]$. 
Together, these imply $[d(1_{C(b-a)})]=[d(1_{X- C(a-b)})]$.
By Lemma~\ref{l:coarse complement for PDn}(3), $C(b-a)\ceq X- C(a-b)$.

Next, we prove~\eqref{2}. We will show that $[a,b,c]\in \mathcal{C}$ iff $C(b-a)\csubset C(c-a)$. The proof of $[a,b,c]\in \mathcal{C}$ iff $C(c-b)\csubset C(c-a)$ will be similar.
    
    We first prove that $[a,b,c]\in \mathcal{C}$ if $C(b-a)\csubset C(c-a)$.
    Note that, $f([o,b)\times \ee^n)\csubset{} f((a,b)\times \ee^n)$.
    By Lemma~\ref{l:coarse complement for PDn}(2), we have $f((a,b)\times \ee^n)\csubset C(b-a)$. Together, these imply $f([o,b)\times \ee^n)\csubset C(b-a)$. Since $C(b-a)\csubset C(c-a)$, we obtain $f([o,b)\times \ee^n)\csubset C(c-a)$. Hence $[a,b,c]\in \mathcal{C}$.

To prove the converse, suppose that $[a,b,c]\in \mathcal{C}$ which means $f([o,b)\times \ee^n)\csubset C(c-a)$.
We want to show that $C(b-a)\csubset C(c-a)$.
To this end, let $S_1$ be the coarse complementary component of $f((a,b)\times \ee^n)$ that does not coarsely contain $f([o,c)\times \ee^n)$. This $S_1$ exists because of the asymmetry of the cyclic order proved earlier. 
To prove $C(b-a)\csubset C(c-a)$, it is enough to show that $S_1\csubset C(c-a)$ and $S_1\ceq C(b-a)$.

We first show that $S_1\csubset C(c-a)$.
By Lemma~\ref{l:coarse complement for PDn}(4), we can assume that $S_1$ is a deep $s$-path component which avoids $N_r(f(a,b)\times\ee^n)$ for any $r$.

Since $f([o,c)\times \ee^n)$ is not coarsely contained in $S_1$, by connectedness of the cyclic order we have $f([o,c)\times \ee^n)\csubset X-S_1$.
Therefore, for any given $r$, we can assume that $S_1$ misses $N_r(f((a,c)\times\ee^n))$.
By Lemma~\ref{l:connected and complementary comp}, either $S_1\subset C(c-a)$ or $S_1\subset X-C(c-a)$.
On the contrary, suppose $S_1\subset X-C(c-a)$. 
Since $f([o,b)\times \ee^n)\csubset S_1$ this would imply $f([o,b)\times \ee^n)\csubset X-C(c-a)$.
Since $f([o,b)\times \ee^n)\csubset C(c-a)$ and $C(c-a)\ccap X-C(c-a)\csubset f((a,c)\times \ee^n)$, it follows that $f([o,b)\times \ee^n)\csubset f((a,c)\times \ee^n)$. This is a contradiction, because the ray $[o,b)$ diverges from $(a,c)$ and $f$ is a coarse equivalence.
Therefore, $S_1\csubset C(c-a)$.

\begin{figure}
\begin{tikzpicture}[scale=1]
    \draw[->, snake it] (0,0) -- (4,0) node[right] {$f([o,b)\times \ee^n)$};
    \draw[->, snake it] (0,0) -- (0,4) node[above] {$f([o,c)\times \ee^n)$};
    \draw[->, snake it] (0,0) -- (0,-4) node[below] {$f([o,a)\times \ee^n)$};
    \node at (2, -2) {$S_1$};
     \node at (2, 2) {$S_2$};
    \filldraw (0,0) circle (1pt) node[below left] {$f(\{o\}\times \ee^n)$};
\end{tikzpicture}
    \label{f:coarse complements}
\end{figure}
Next, we show that $S_1\ceq C(b-a)$.
To this end, we first let $S_2$ to be the coarse complementary component of $f((b,c)\times \ee^n)$ that does not coarsely contain $f([o,a)\times \ee^n)$.
By a similar argument used to prove $S_1\csubset C(c-a)$, we can prove $S_2\csubset C(c-a)$ and henceforth, we will assume $S_2\csubset C(c-a)$.
We first claim that to prove $S_1\ceq C(b-a)$, it is enough to show that $[d(1_{S_1})]+[d(1_{S_2})]=[d(1_{C(c-a)})]$ as a class in $\Hx[1](X-f(Y))$.
 
Suppose $[d(1_{S_1})]+[d(1_{S_2})]=[d(1_{C(c-a)})]$. Since $\phi(b-a)+\phi(c-b)=\phi(c-a)$, we have $[d(1_{C(b-a)})]+[d(1_{C(c-b)})]=[d(1_{C(c-a)})]$.
Combining, we obtain
\begin{align*}\label{eq}
[d (1_{S_1})]+[d (1_{S_2})]=[d(1_{C(b-a)})]+[d(1_{C(c-b)})]. \tag{$\dagger$}
\end{align*}
Since $S_1$ and $S_2$ are coarse complementary components of $f((a,b)\times \ee^n)$ and $f((b,c)\times \ee^n)$, we know that $[d(1_{S_1})]$ and $[d(1_{S_2})]$ are the same as $[d(1_{(C(b-a))}]$ and $[d(1_{(C(c-b))}]$ respectively, up to sign by Lemma~\ref{l:coarse complement for PDn}(1).
It follows that $[d(1_{S_1})]$, $[d(1_{S_2})]$, and their sum are infinite order elements.
Therefore, $[d(1_{S_1})]=-[d(1_{C(b-a)})]$ or $[d(1_{S_2})]=-[d(1_{C(c-b)})]$ violate ~\eqref{eq}.
Therefore, we conclude
$[d(1_{S_1})]= [d(1_{C(b-a)})]$ and $[d(1_{S_2})]= [d(1_{C(c-b)})]$.
By Lemma~\ref{l:coarse complement for PDn}(3), we have $S_1\ceq C(b-a)$ and $S_2\ceq C(c-b)$.

It remains to show that $[d(1_{S_1})]+[d(1_{S_2})]=[d(1_{C(c-a)})]$.
Recall that we can assume $S_1$ to be a deep $s$-path connected component of $X-N_r(f(a,b)\times\ee^n)$ for any $r$.
Note that $S_1$ is not a subset of $S_2$ because $f([o,a)\times \ee^n)$ coarsely contained in $S_1$ but not in $S_2$ by construction. 
Therefore, by Lemma~\ref{l:connected and complementary comp}, we have $S_1\subset X-S_2$.
 Therefore from now on, we assume that $S_1\cap S_2=\emptyset$. 

Coming back to the proof of $[d(1_{S_1})]+[d(1_{S_2})]=[d(1_{C(c-a)})]$ with $S_1\cap S_2=\emptyset$ in hand, we first observe that
\[
d(1_{S_1})+d(1_{S_2})-d(1_{C(c-a)})=d(1_{S_1\cup S_2})-d(1_{C(c-a)})=d(1_{S_1\cup S_2}-1_{C(c-a)}).
\]
Note that, it suffices to prove $1_{S_1\cup S_2}-1_{C(c-a)}\in \Cx[0](X-f(Y))$. In other words, it suffices to prove that $(S_1\cup S_2)\symdiff C(c-a)\csubset f(Y)$.
More specifically, we will show that $(S_1\cup S_2)\symdiff C(c-a)\csubset A$ where $A=f(\triangle(abc)\times \ee^n)$ and $\triangle(abc)$ is union of rays $(a,b)$, $(b,c)$ and $(c,a)$ in $T$.

To show this,
let $q:\zz\rightarrow\zz/2 = \langle \mathbf{1} \rangle$ be the quotient map and let $q':\Hx(X-f(Y);\zz)\to \Hx(X-f(Y);\zz/2)$ be the induced homomorphism. 
Recall that $[d(1_{S_1})]$, $[d(1_{S_2})]$, and $[d(1_{C(c-a)})]$ are up to sign the same as $\phi(b-a)$, $\phi(c-b)$, and $\phi(c-a)$, respectively.
It follows that $[d(\mathbf{1}_{S_1})]$, $[d(\mathbf{1}_{S_2})]$, and $[d(\mathbf{1}_{C(c-a)})]$ are same as $q'(\phi(b-a))$, $q'(\phi(c-b))$, and $q'(\phi(c-a))$, respectively. 
Consequently, we have
\[
 [d(\mathbf{1}_{S_1\cup S_2})]=[d (\mathbf{1}_{S_1})]+[d(\mathbf{1}_{S_2})]=q'(\phi(b-a))+q'(\phi(c-b))=q'(\phi(c-a))=[d(\mathbf{1}_{C(c-a)})].
\]
By Lemma~\ref{l:coarse complement}, either $(S_1\cup S_2)\symdiff C(c-a)\csubset A$ or $(S_1\cup S_2)\symdiff (X-C(c-a))\csubset A$.
Since $S_1\cup S_2\csubset C(c-a)$, $(X-C(c-a))-(S_1\cup S_2)$ coarsely contains $X-C(c-a)$. 
Since $X-C(c-a)$ does not coarsely contains $A$, we conclude that $(S_1\cup S_2)\symdiff (X-C(c-a))\csubset A$ is not possible.
Hence $(S_1\cup S_2)\symdiff C(c-a)\csubset A$.
This finishes the proof.
\end{proof}

\subsection{Respecting the cyclic order}

We now show that the cyclic order $\mathcal{C}$ is respected by any map $g:T\times \ee^n\rightarrow T\times \ee^n$ that coarsely extends to a coarse equivalence $\bar{g}:X\to X$.
For convenience, the restriction of $g$ on the $T$ factor will be also denoted by $g$.

Let $\delta(X)\subset X\times X$ be the diagonal subset.
Since $X$ is a coarse $\mathrm{PD}(n)$ space, it follows that 
\[\Hx[n](X\otimes X-\delta(X))=\Hx[n](X\times X-\delta(X))=\hx[n](\delta(X))=\zz.
\]
Furthermore, since $\bar{g}:X\to X$ is a coarse equivalence, the induced map $$\bar{g}^*:\Hx[n](X\otimes X-\delta(X))\to \Hx[n](X\otimes X-\delta(X))$$ is an isomorphism $\zz\rightarrow \zz$. We claim that $g$ preserves the cyclic order if $\bar{g}^*$ is identity and reverses the cyclic order otherwise.

Let $\Phi$ be the following slant product map.
\[\begin{tikzcd}
     \Hx[n](X\otimes X-\delta(X)) \times  \hx[n+1](f(Y)) \to   \Hx[1](X-f(Y)).
\end{tikzcd}
\]
Since $\hx[n+1](f(Y))$ is isomorphic to $\rh[0](\partial_\infty T)$, we can replace $\hx[n+1](f(Y))$ by $\rh[0](\partial_\infty T)$ and consider $\Phi$ as the following map
\[\begin{tikzcd}
     \Hx[n](X\otimes X-\delta(X)) \times \rh[0](\partial_\infty T) \to   \Hx[1](X-f(Y)).
\end{tikzcd}
\]

Let $\mathcal{E}\in\Hx[n](X\otimes X-\delta(X))$ such that  $\Phi(\mathcal{E}, c-a)=\phi(c-a)$ for all $a,c\in \partial_\infty T$.
Such $\mathcal{E}$ exists due to the second part of the Theorem~\ref{CAD}.
Furthermore, we note that 
\[
\Phi(\bar{g}^*(\mathcal{E}),c-a)=\bar{g}^*(\Phi(\mathcal{E}, g(c)-g(a))).
\]

To prove the claim, suppose $\bar{g}^*$ is identity. Therefore, $\bar{g}^*(\mathcal{E})=\mathcal{E}$ and consequently, we have
$\phi(c-a)=\bar{g}^*(\phi(g(c)-g(a)))$.
It follows that $[d(1_{\bar{g}(C(c-a))})]=[d(1_{C(g(c)-g(a))})]$.
By Lemma~\ref{l:coarse complement for PDn}(3), we have $\bar{g}(C(c-a))\ceq C(g(c)-g(a))$.
Note that $f([o,b)\times \ee^n)\csubset C(c-a)$ implies $\bar{g}(f([o,b)\times \ee^n))\csubset \bar{g}(C(c-a))$.
Hence, $\bar{g}(f([o,b)\times \ee^n)) \csubset C(g(c)-g(a))$.
Since $\bar{g}(f([o,b)\times \ee^n))\ceq f([o,g(b))\times \ee^n)$, we conclude that $f([o,g(b))\times \ee^n)) \csubset C(g(c)-g(a))$.
In other words, $[g(a),g(b),g(c)]\in \mathcal{C}$.

Suppose $\bar{g}^*$ is not identity. Then $\bar{g}^*(\mathcal{E})=-\mathcal{E}$. 
Arguing as before, we obtain $-\phi(c-a)=\bar{g}^*(\phi(g(c)-g(a)))$.
It follows that $\bar{g}(C(c-a))\ceq C(g(a)-g(c))$.
Arguing as before, we can conclude that $f([o,g(b))\times \ee^n) \csubset C(g(a)-g(c))$ which implies $[g(c),g(b),g(a)]\in \mathcal{C}$.

This completes the proof of Proposition~\ref{p:cyclic order high dimension}.

 \begin{Remark}
 The only property of $\ee^n$ used in the proof of Proposition~\ref{p:cyclic order high dimension} is that it is a coarse $\mathrm{PD}(n)$ space. Therefore, Proposition~\ref{p:cyclic order high dimension} remains valid if  $\ee^n$ is replaced by any coarse $\mathrm{PD}(n)$ space.
 \end{Remark}

 \subsection{Relating the visual topology of $\partial_\infty T$ to the cyclic order}Associated to the cyclic order $\mathcal{C}$ obtained in Proposition~\ref{p:cyclic order high dimension}, we can define intervals in $\partial_\infty T$ as follows. 

 For $a,b\in \partial_\infty T$, we let 
 $(a,b)_\mathcal{C}:=\{x\mid [a,x,c]\in \mathcal{C}\}$
and $[a,b]_{\mathcal{C}}:=(a,b)_\mathcal{C}\cup \{a,b\}$.

The following two lemmas, which will be used in the next section, assert that the cyclic order $\mathcal{C}$ is, in a rough sense, compatible with the visual topology on $\partial_\infty T$.

 \begin{Lemma}\label{l:interval is open}
     $(a,b)_\mathcal{C}$ is open in $\partial_\infty T$ for any $a,b\in \partial_\infty T$.
\end{Lemma}

\begin{proof}
Suppose $y\in (a,b)_\mathcal{C}$, which means $f([o,y)\times \ee^n)\csubset C(b-a)$.
    We want to show that $y$ has an open neighborhood contained in $(a,b)_\mathcal C$.
    Suppose $\{y_m\}$ is a sequence in $\partial_\infty T$ such that $\lim y_m= y$. It is enough to show that $y_m\in (a,b)_\mathcal{C}$ eventually.

Suppose $\gamma:[0,\infty)\rightarrow T$ and $\gamma_m:[0,\infty)\rightarrow T$ represent the geodesic rays $[o,y)$ and $[o,y_m)$, respectively.

    Since $f(\gamma([0,\infty))\times \ee^n)\csubset C(b-a)$ band $[d(1_{N_l(C(b-a)}))]=[d(1_{C(b-a)})]$ for any $l>0$, we can assume $f(\gamma([0,\infty))\times \ee^n)\subset C(b-a)$ by replacing $C(b-a)$ with  $N_l(C(b-a))$ for some $l$.

 Since $y_m\to y$, there exists a sequence  $\{t_m\}$ in $[0,\infty)$ such that $t_m\to \infty$, and
   \[d(\gamma(t_m),\gamma_m(t_m))\leq 1 \quad \text{eventually.}
   \]

    Fix $e\in \ee^n$.
    Since $f$ is a coarse embedding,
    there exists an $l$ such that 
    \[d(f(\gamma(t_m),e),f(\gamma_m(t_m), e))\leq l \quad  \text{eventually.}
    \]
Since $f(\gamma([0,\infty)\times \{e\})\subset C(b-a)$, we obtain that $f(\gamma_m(t_m),e)\in N_l(C(b-a))$ eventually.
 By replacing $C(b-a)$ by $N_l(C(b-a))$, without loss of generality, we can assume $f(\gamma_m(t_m), e)\in C(b-a)$ eventually.

Note that there exists $s$, such that $f(\gamma_m([r,\infty))\times \ee^n)$ is $s$-connected for all $r$.  
By Lemma~\ref{l:connected and complementary comp}, it follows that there exists $R$, such that $f(\gamma_m([R,\infty))\times \ee^n)\subset C(b-a)$ or $f(\gamma_m([R,\infty))\times \ee^n)\subset X-C(b-a)$ for all $m$.
Since $f(\gamma_m(t_m), e)\in C(b-a)$ eventually and $t_m\to \infty$, it follows that $f(\gamma_m([R,\infty))\times \ee^n)\subset C(b-a)$ eventually.
Since  $f(\gamma_m([0,\infty))\times \ee^n)\csubset f(\gamma_m([R,\infty))\times \ee^n)$, we obtain $f([o,y_m)\times \ee^n)=f(\gamma_m([0,\infty))\times \ee^n)\csubset C(b-a)$ eventually.
Hence, $y_m\in (a,b)_\mathcal{C}$ eventually.
\end{proof}

\begin{Lemma}\label{l:small intervals}
    Suppose $V\subset \d_\infty T$ be an open set and  $x\in V$. Suppose $\{x_n\}$ is a sequence in $\partial_\infty T$ that converges  to $x$. Then there exists $K$ such that for all $k\geq K$, either $[x_k,x]_\mathcal{C}\subset V$ or $[x,x_k]_\mathcal{C}\subset V$.
\end{Lemma}
\begin{proof}
    Consider the geodesic $(x_k,x)$ in $T$ with endpoints in $x_k$ and $x$.
    Let $f:T\times \ee^n\to X$ be the coarse embedding.
    Let $o\in T$ be a base point and $[o,z)$ denote the geodesic ray connecting $o$ and $z\in \d_\infty T$.
    For any $S$, there exists $K$, such that $d((x_k,x),[o,z))\geq S$ for all $k\geq K$ and $z\in \d_\infty T - V$.
    Since $f$ is a coarse embedding, it follows that for any $S$ there exists $K$ such that $d(f((x_k,x)\times \ee^n), f([o,z)\times \ee^n))\geq S$ for all $k\geq K$ and $z\in \d_\infty T - V$.
    It follows that, for large enough $k$, there exists a coarse complementary component $C_k$ of $f((x_k,x)\times \ee^n))$ such that  $f([o,z)\times \ee^n)\csubset C_k$ for all $z\in \d_\infty T- V$.
    Therefore, $\partial_\infty T- V\subset (x_k,x)_\mathcal{C}$ or $\partial_\infty T- V\subset (x,x_k)_\mathcal{C}$ for large enough $k$.
    Taking complements, we obtain that for large enough $k$, either $[x_k,x]_\mathcal{C}\subset V$ or $[x,x_k]_\mathcal{C}\subset V$.
\end{proof}

 \section{Incompatible group action}\label{s:incompatible group action}
 
 In this section, we prove Proposition~\ref{t:Incompatible group action}. 
Recall from the introduction that if $L$ is a finitely generated free abelian group with a chosen basis, $A\in GL(n,\mathbb{Q})$, and $L'$ a finite-index subgroup of $L\cap A^{-1} L=\zz^n\cap A^{-1}L$, then we write $G(L,A,L')$ for the HNN extension \[
G(L,A,L'):=\langle L,t\mid tct^{-1}=A(c),\forall c\in L' \rangle.
\]

 By Theorem~\ref{t:infinite index} the group $G=G(L,A,L')$ acts on $T\times \ee^n$ ($T$ is the assocated Bass-Serre tree) by isometries when the matrix $A$ is conjugate in $GL(n,\rr)$ to an orthogonal matrix.
The next lemma says that there exist elements in $G$ whose action on the tree factor $T$ are arbitrarily close to the identity map given that $A$ has infinite order.
\begin{Lemma}\label{l:action on tree}
    Suppose $A$ is conjugate in $GL(n,\rr)$ to an orthogonal matrix and has infinite order. Then there exists an element $a\in L$ with the following properties:
    \begin{itemize}
    \item There exists a point $x\in \partial_\infty T$ such that $a^kx\neq x$ for all $k\neq 0$.
        \item There exists a point $v \in T$ and an increasing sequence of natural numbers $\{n_i\}$ such that the action of $\langle a^{n_i} \rangle$ on the tree factor $T$ stabilizes the ball of radius $i$ around $v$. 

    \end{itemize}
\end{Lemma}
\begin{proof}
Note that since $L'$ and $A(L')$ are finite index in $L$, the Bass-Serre tree $T$ is locally finite. 
The stabilizer of the vertices of $T$ are conjugates of the vertex group $L$.  
Therefore, the subgroup that fixes the entire tree is the intersection of all conjugates of $L$, denoted by $c(L)$.
It follows from the proof of~\cite[Theorem 7.5]{Leary-Minasyan} that if $A$ has infinite order, then $|L:c(L)|$ is infinite. 
Since $L$ is a  finitely generated abelian group, there is an element $a\in L$ that represents an infinite order element in the quotient $L/c(L)$.
Consequently, $a$ represents an infinite order element in $\Isom{(T)}$ that stabilizes a vertex.

Take a vertex $v\in T$ such that $\Stab(v)=L$.
Let $t\in G(A,L')$ be the stable letter.

For the first claim, we let $x,x'\in \partial _\infty T$ be the end points of the geodesic with vertices $\{\ldots, t^{-2}v,t^{-1}v,v,tv,t^2v,\ldots\}$.
We claim that $a^kx\neq x$ for all $k\neq 0$ or $a^kx'\neq x'$ for all $k\neq 0$.
If not, then some $a^k$ fixes both  $x$ and $x'$, and since it fixes $v$ it fixes this geodesic pointwise. 
Since $\Stab(v)=L$, we have $\Stab(t^iv)=t^iLt^{-i}$.
It follows that $a^k\in t^iLt^{-i}$ for all $i$ and therefore $a^k\in c(L)$. This is a contradiction. Hence, the claim is true.

For the second claim, note that $\langle a \rangle$ stabilizes all the balls around $v$.
It follows that for each $i$, there exists $n_i\in \nn$, such that $\langle a^{n_i} \rangle$ acts trivially on the ball of radius $i$ around $v$ in $T$.
The claim follows.
\end{proof}

We now proceed to prove Proposition~\ref{t:Incompatible group action}. First, we recall the statement.
\begin{Prop}\label{p:incompatible action}
Suppose that $G=G(A,L')$ where $A$ has infinite order and is conjugate
in $GL(n,\rr)$ to an orthogonal matrix. If $\Gamma$ is a finite index subgroup of $G(A,L')$, then the action of $\Gamma$ on $T$ does not respect the cyclic order $\mathcal{C}$.
\end{Prop}

\begin{proof}

Let $a\in G(A,L')$ be as in Lemma~\ref{l:action on tree}.
Since $\Gamma$ is a finite index subgroup, $a^l\in \Gamma$ for some $l\in \nn$.
We let $\beta:=a^l$.
We aim to show that $\beta$ does not respect $\mathcal{C}$.
On the contrary, suppose each element in $\langle \beta\rangle$ respects $\mathcal{C}$.
Taking the square of $\beta$ if necessary, we can assume that $\langle\beta\rangle$ preserves the cyclic order.
Using Lemma~\ref{l:action on tree}, we pick $x\in \partial_\infty T$ such that  $\beta^kx\neq x$ for any $k\neq 0$.

By Lemma~\ref{l:action on tree}, we can choose an open set $V$ containing $x$ so that the union of $\langle \beta^k\rangle$ translates of $V$ does not cover $\d_\infty T$ for large enough $k$.
By Lemma~\ref{l:action on tree}, some subsequence of $\{\beta^{n_k}x\}_{k\in \nn}$ converges to $x$.
By Lemma~\ref{l:small intervals}, we have $[\beta^{n_k}x,x]_\mathcal{C}\subset V$ or $[x,\beta^{n_k}x]_\mathcal{C}\subset V$ for large enough $k$.
In other words, either $[x,\beta^kx]_\mathcal{C}\subset V$ or $[\beta^{k}x,x]_\mathcal{C}\subset V$ for some arbitrarily large $k$.

 We first consider the case where  $[x,\beta^kx]_\mathcal{C}\subset V$ for some arbitrarily large $k$.
We consider the  following set \[Y_k:=\cup_{i=0}^\infty [\beta^{ik}x,\beta^{(i+1)k}x]_\mathcal{C}.
\]
Since $\langle \beta \rangle$ preserves the cyclic order, $Y_k$ is the union of the $\langle \beta^k\rangle$-translates of $[x,\beta^k x]_\mathcal{C}$.
Since $\langle \beta^k\rangle$-translates of $V$ does not cover $\d_\infty T$ for large enough $k$,
it follows that $Y_k\neq \partial_\infty T$ for some large $k$.

We take $z\in \partial_\infty T- Y_k$ which means $[\beta^{ik}x,\beta^{(i+1)k}x,z]\in \mathcal{C}$ for all $i\in \nn\cup \{0\}$.
By transitivity of the cyclic order, we have $[\beta^{k}x,\beta^{ik}x,z]\in \mathcal{C}$ for all $i\geq 2$. In other words, $\beta^{ik}x\in{} [\beta^kx,z]_\mathcal{C}$ for all $i\geq 2$.
By asymmetry of the cyclic order, we have $\beta^{ik}x\notin{} (z,\beta^k x)_\mathcal{C}$ for all $i\geq 2$.
Additionally, since $[x,\beta^kx,z]\in \mathcal{C}$, we have $x\in (z,\beta^k x)_\mathcal{C}$ by cyclicity of $\mathcal{C}$.
By Lemma~\ref{l:interval is open}, the interval $(z,\beta^k x)_\mathcal{C}$ is an open set.
It follows that the sequence $\{\beta^{ik}x\}_{i\in \nn}$ does not have any subsequence that converges to $x$.
Since Lemma~\ref{l:action on tree} implies that there is an increasing sequence $\{n_i\}$  such that $\beta^{kn_i}x\rightarrow x$ as $i\rightarrow \infty$, we arrive at a contradiction.

So we are left with the case where $[\beta^{k}x,x]_\mathcal{C}\subset V$ for some arbitrarily large $k$.
Applying $\beta^{-k}$, we get $[x, \beta^{-k}x]_\mathcal{C}\subset \beta^{-k}V$ for some large $k$.
Since the $\langle \beta^k \rangle$ translate of $V$ do not cover $\d_\infty T$ for large $k$, it follows that $Y_{-k}\neq \d_\infty T$ for some large $k$.
We can now apply the same argument as before, replacing $k$ by $-k$, to get a contradiction.
\end{proof}

This completes the proof of Theorem~\ref{t:main theorem}.

\bibliography{bibliography.bib}
\bibliographystyle{amsalpha} 
\end{document}